\documentclass[10 pt]{article}
\usepackage{amssymb}
\usepackage{amsthm}
\usepackage{amsmath}
\usepackage{mathrsfs}
\usepackage{verbatim}
\usepackage{color}
\usepackage{changes}
\usepackage{color}

\newcommand{\R}{\mathbb{R}}

\begin{document}

\title{Decoupling of DeGiorgi-type systems via multi-marginal optimal transport\footnote{BP is pleased to acknowledge the support of a University of Alberta start-up grant. The research of both authors was supported by the Natural Sciences and Engineering Research Council of Canada.  BP would like to thank Almut Burchard for a very useful discussion on rearrangement inequalities.}}\author{Nassif Ghoussoub\footnote{Department of Mathematics, University of British Columbia, Vancouver, BC, Canada, V6T 1Z2 nassif@math.ubc.ca} \qquad and \qquad Brendan Pass\footnote{Department of Mathematical and Statistical Sciences, 632 CAB, University of Alberta, Edmonton, Alberta, Canada, T6G 2G1 pass@ualberta.ca.}}
\maketitle

\begin{abstract}
We exhibit a surprising relationship between elliptic gradient systems of PDEs, multi-marginal Monge-Kantorovich optimal transport problem, and multivariable Hardy-Littlewood inequalities. We show that the notion of an {\it orientable} elliptic system, conjectured  in \cite{FG} to imply that (in low dimensions) solutions with certain monotonicity properties are essentially $1$-dimensional, is equivalent to the definition of a \textit{compatible} cost function, known to imply uniqueness and structural results for optimal measures to certain Monge-Kantorovich problems \cite{P4}.  Orientable nonlinearities and  compatible cost functions  turned out to be also related to submodular functions, which appear in rearrangement inequalities of Hardy-Littlewood type.
We use this equivalence to 
establish a decoupling result for certain solutions to elliptic PDEs and show that under the orientability condition, the decoupling has additional properties, due to the connection to optimal transport.
\end{abstract}

\section{Introduction}
The main purpose of this note is to pinpoint a surprising connection between elliptic systems of PDEs,  
 multi-marginal optimal transportation, and multivariable extended Hardy-Littlewood inequalities.
A recent paper by Fazly and Ghoussoub \cite{FG} introduced the concept of an \textit{orientable} elliptic system (Definition \ref{hmono} below), which seems to be the appropriate framework for investigating De Giorgi type conjectures (\cite{AC}), \cite{GG}, \cite{GG2}) for systems of more than two equations.  On the other hand, the thesis of the second author \cite{P4}, following work of Carlier \cite{C}, introduced the concept of a {\it compatible} cost function (Definition \ref{hcomp} below), a natural, covariant condition ensuring uniqueness and structural results on solutions to a multi-marginal optimal transportation problem with one dimensional marginals. These notions turned out to be also related to submodular (or $2$-monotone) functions, which appear in rearrangement inequalities of Hardy- Littlewood type as studied by several authors dating back to Lorentz \cite{lorentz}.

We will show here that these three conditions are actually equivalent.  As a consequence we shall see how the concept of an $H$-monotone solution to the system \eqref{ellsys} below, which was  introduced in \cite{FG}, is intimately related to the geometric structure of optimal measures in the optimal transport problem \eqref{multmarg} uncovered in \cite{P4}. 
We also show that monotone solutions to the elliptic system can be {\it decoupled}.  If the solution is $H$-monotone, we use the connection with optimal transportation to show that the sum of the decoupled non-linearities is everywhere less than the original non-linearity.

Let $H: \mathbb{R}^m \rightarrow \mathbb{R}$ be a $C^2$ function.  We will consider the system of elliptic PDEs on $\mathbb{R}^N$:
\begin{equation}\label{ellsys}
\Delta u = \nabla H (u),
\end{equation}
where $u=(u_1,u_2,...,u_m): \mathbb{R}^N \rightarrow \mathbb{R}^m$ represents an $m$-tuple of functions.  In this context, $H$ is often referred to as the {\it non-linearity} of the system. 

Given probability measures $\mu_1,\mu_2,...,\mu_m$ on $\mathbb{R}$, the optimal transport (or Monge-Kantorovich) problem consists of minimizing 
\begin{equation}\label{multmarg}
\int_{\mathbb{R}^m} H(p_1,p_2,...,p_m)d\gamma(p_1,p_2,...,p_m)
\end{equation}
among probability measures $\gamma$ on $\mathbb{R}^m$ whose $1$-dimensional marginals are $\mu_i$.  In this setting, $H$ is called the {\it cost function}.

If $H$ is bounded below on $\mathbb{R}^m$, then there exists a solution $\bar\gamma$ to the Kantorovich problem (\ref{multmarg}), as well as an $m$-tuple of functions $(V_1,V_2,...,V_{m})$ --  called Kantorovich potentials -- such that for all $i=1, ..., m,$
\begin{equation}
V_i(p_i)=\inf_{\substack{p_j \in \mathbb{R}\\ j\neq i}}\Big( H(p_1,p_2,...,p_{m})-\sum_{j \neq i}V_j(p_j)\Big),
\end{equation}
and which maximizes the following dual problem
\begin{equation}\label{dual}
\sum^{m}_{i=1}\int_{\mathbb{R}} V_i(p_i)d\mu_i 
\end{equation}
among all $m$-tuples $(V_1, V_2,...,V_{m})$ of functions $V_i \in L^1(\mu_i)$ for which 
\begin{equation}
\hbox{$\sum_{i=1}^{m} V_i(p_i ) \leq H(p_1,...,p_{m})$ for all $(p_1,...,p_{m}) \in  \mathbb{R}^m$}
\end{equation}  
(see, for example, Theorem 4.1.1 in \cite{P4}).
 Furthermore, the maximum value in (\ref{dual}) coincides with the minimum value in (\ref{multmarg}) and    \begin{equation}
\hbox{$  \sum\limits_{i=0}^{m-1}V_i(p_i) =H(p_1,...,p_{m})$ \quad  for all  $(p_1,...,p_{m}) \in {\rm support}(\bar\gamma)$.}
  \end{equation}
 In a certain sense, the dual problem (\ref{dual}) provides a decoupling of the original Monge-Kantorovich problem. In this note, we shall show that the above scheme also provides a decoupling of the system (\ref{ellsys}), at least for certain type of solutions. Roughly speaking, under certain monotonicity conditions on a solution $u=(u_1, u_2,..., u_m)$ of (\ref{ellsys}), the duality in the Monge-Kantorovich problem  applied to a suitable set of marginals $(\mu_{u_1}, \mu_{u_2},..., \mu_{u_2})$ associated to $u$, leads to a decoupled system 
 \begin{equation}
\hbox{$\Delta u_i = \frac{\partial V_i}{\partial p_i}(u_i(x))$\,\, for $i=1,..., m.$}
\end{equation}
having $u=(u_1, u_2,..., u_m)$ as a solution, where $V_1,..., V_m$ are the corresponding Kantorovich potentials. 
Decoupled systems are much simpler than coupled systems and have a number of advantages.  For example, whenever the decoupling above is possible and the potentials $V_i$ are positive, one can deduce the following Modica inequality \cite{M} for systems
\begin{equation}
\hbox{$\sum_{i=1}^m|\nabla u_i(x)|^2 \leq 2H(u_1(x),..., u_m(x))$ for all $x\in \mathbb{R}^N$.}
\end{equation}
 Whether such an inequality holds true for a general gradient system remains an open problem. See Alikakos \cite{Ali}. 

The compatibility condition is also equivalent --up to a change of variables-- to the classical notion of submodularity (also known as $2$-monotonicity) of $H$ (see Definition \ref{submod}).   A result of Carlier on the structure of optimizers in \eqref{multmarg} for a submodular cost function $H$ is essentially equivalent to a rearrangement inequality of Hardy-Littlewood type \cite{C}\cite{lorentz}\cite{BH}.  It follows that the covariant analogue of this result noted  in \cite{P4} implies a rearrangement inequality for compatible $H$,  but where decreasing rearrangements should be replaced by {\it $H$-monotone rearrangements}.  Note that rearrangement inequalities have found applications in nonlinear optics \cite{HS}, where one looks for ground states of energy functionals  of the form 
\begin{equation}\label{energy}
E(u) = \int_{\mathbb{R}^N} \frac{1}{2} \sum_{i=1}^m |\nabla u_i(x)|^2 +H(u) dx.
\end{equation}
Note that  \eqref{ellsys} is nothing but the Euler-Lagrange system corresponding to this energy functional.

For a submodular $H$,  any nonnegative function $u$ may be rearranged in a symmetric and decreasing way (that is, the decreasing rearrangement is with respect to the variable $|x|$)  to decrease the total energy.  In our case, at least on bounded domains, where $H$ is orientable, any $u$ can be rearranged in an $H$-monotone way with respect to the variable $x_N$ to decrease the second component of the energy.  On appropriate domains rearrangement will also decrease the gradient term, so that the total energy also decreases under $H$-monotone rearrangement (see Theorem \ref{bounded_domains} below).

Finally, we note that, while these types of decouplings can only be done for gradient systems, it is not essential to have the Laplacian on the left hand side; similar decouplings are possible when the left hand side is replaced by any decoupled differential operator $D_iu_i$.  

\section{Orientable systems, compatible cost and sub-modular functions}
We begin by recalling the definitions of an orientable system and of a compatible cost.  Our definitions here are actually strict versions of the original definitions in \cite{FG}, where non-strict inequalities are used. Note that in what follows, we do not use an implicit summation convention.
\newtheorem{hmono}{Definition}[section]

\begin{hmono}\label{hmono}
The system (or the non-linearity $H: \R^m \rightarrow \R$) is called orientable on a subset $\Omega$ of $\R^m$, if there exist constants $(\theta_i)_{i=1}^m$ such that, for all $i \neq j$,
\begin{equation*}
\hbox{$\theta_i\theta_j \frac{\partial^2 H}{\partial p_i \partial p_j}({\bf p}) <0$  for all ${\bf p}=(p_1,..., p_m)\in \Omega$.}
\end{equation*}
\end{hmono}

Next, we recall the definition of compatibility, discussed in \cite{P4}.
\newtheorem{hcomp}[hmono]{Definition}
\begin{hcomp}\label{hcomp}
We say $H$ is compatible on $\Omega$ if for all distinct $i,j,k$, we have 

\begin{equation*}
\hbox{$\frac{\partial ^2 H}{\partial p_i \partial p_j} \Big(\frac{\partial ^2 H}{\partial p_k \partial p_j}\Big)^{-1} \frac{\partial ^2 H}{\partial p_k \partial p_i}<0$ for all ${\bf p} =(p_1,p_2,...,p_m) \in \Omega$.}
\end{equation*}

\end{hcomp}

Next, recall the following classical definition.
\newtheorem{submod}[hmono]{Definition}
\begin{submod}\label{submod}
The function $H$ is submodular (or $2$-increasing in Economics)  on $\Omega$ if
\begin{equation*}
\hbox{$H({\bf p}+h{\bf e}_i+k{\bf e}_j)+H({\bf p})- H({\bf p}+h{\bf e}_i)-H({\bf p}+k{\bf e}_j) \leq 0 $ \qquad $(i\neq j, \quad h, k >0),$}
\end{equation*}
where ${\bf p}=(p_1,...,p_m) \in \Omega$ and ${\bf e}_i$ denotes the $i$-th standard basis vector in ${\mathbb{R}^m}$. 
\end{submod}
\newtheorem{equivcond}[hmono]{Lemma}
\begin{equivcond}\label{equivcond} The following are equivalent for a function $H\in C^2({\mathbb R}^m, \mathbb{R})$.
\begin{enumerate}
\item $H$ is orientable on $\Omega$.
 \item $H$ is compatible  on $\Omega$.
 \item After a change of variables, $H$ is submodular  on $\Omega$.
 \end{enumerate}
\end{equivcond}
\begin{proof}
We will first show that 1) and 2) are equivalent. First suppose $H$ is orientable.  Then we have, for all distinct $i,j,k$,
\[
\frac{\partial ^2 H}{\partial p_i \partial p_j} \theta_i\theta_j\Big(\frac{\partial ^2 H}{\partial p_k \partial p_j}\theta_k\theta_j\Big)^{-1} \frac{\partial ^2 H}{\partial p_k \partial p_i}\theta_k\theta_i<0,
\]
and therefore
\[
\theta_i^2\frac{\partial ^2 H}{\partial p_i \partial p_j} \Big(\frac{\partial ^2 H}{\partial p_k \partial p_j}\Big)^{-1} \frac{\partial ^2 H}{\partial p_k \partial p_i}<0,
\]
which implies compatibility.

On the other hand, if $H$ is compatible, set $\theta_i =1$ if $\frac{\partial ^2 H}{\partial p_i \partial p_1} >0$ and $\theta_i =-1$ if $\frac{\partial ^2 H}{\partial p_i \partial p_1} <0$.  Then, up to positive multiplicative constants, we have

\begin{equation*}
\theta_i\theta_j\frac{\partial ^2 H}{\partial p_i \partial p_j} =\frac{\partial ^2 H}{\partial p_i \partial p_1}\frac{\partial ^2 H}{\partial p_1 \partial p_j}\frac{\partial ^2 H}{\partial p_i \partial p_j} <0,
\end{equation*}
establishing $H$-orientability.

Now we will prove the equivalence of 2) and 3). First, recall that a smooth function $H$ is submodular if 
\begin{equation}\label{smooth}
\hbox{$\frac{\partial ^2 H}{\partial p_i \partial p_j}<0$ for all $i \neq j$ and $p =(p_1,p_2,...,p_m)$.} 
\end{equation}
As was noted in \cite{P4}, compatibility is equivalent to the existence of changes of variables $p_i \mapsto q_i$ such that $H(q_1,q_2,...,q_m)$ satisfies (\ref{smooth}). It is therefore invariant under this sort of transformation. 
Assuming now that $H$ is compatible, define a change of coordinates as follows: set $q_1=p_1$ and, for $i \geq 2$, set 
\begin{eqnarray}\label{chcoord}
q_i = p_i & &\text{if } \frac{\partial^2 H}{\partial p_1 \partial p_i}<0\\ 
q_i = -p_i& & \text{if } \frac{\partial^2 H}{\partial p_1 \partial p_i}>0.  
\end{eqnarray}
It is then clear that  $\frac{\partial^2 H}{\partial q_1 \partial q_i}<0$ for all $i$.  For any distinct $i,j \neq 1$, we then have by the compatibility condition,
\begin{equation*}
\frac{\partial ^2 H}{\partial q_i \partial q_j} \Big(\frac{\partial ^2 H}{\partial q_1 \partial q_j}\Big)^{-1} \frac{\partial ^2 H}{\partial q_1 \partial q_i}<0 ,
\end{equation*}
which easily yields $\frac{\partial ^2 H}{\partial q_i \partial q_j} <0$.

On the other hand, to see that 3) implies 2), it is sufficient to note that submodularity implies compatibility, and that compatibility is invariant under changes of coordinates of the form $p_i \mapsto q_i(p_i)$. \end{proof}

We note that the condition of submodularity of $H$ on the positive half-space ${\R}_+^m$ is essentially equivalent to the following extended Hardy-Littlewood inequality: for all choices of real-valued non-negative measurable functions $(u_1,..., u_m)$ that vanish at infinity, we have
\begin{equation}\label{HL}
\int_{\mathbb{R}^N}H(u_1^*(x),...,u_m^*(x))dx \leq \int_{\mathbb{R}^N}H(u_1(x),..., u_m(x))dx, 
\end{equation}
where $u_i^*$ is the symmetric decreasing rearrangement of $u_i$ for $i=1,..., m$. The reason for this is intuitively clear; the submodularity condition ensures that moving weight from points of the form $ {\bf p}+h{\bf e}_i$ and ${\bf p}+k{\bf e}_j$ to ${\bf p}$ and $h{\bf e}_i+k{\bf e}_j+{\bf p}$ decreases the integral of $H(u)$ whenever $h,k >0$.  The symmetric decreasing rearrangment does precisely this, moving weight onto a monotone set without changing the level sets of the $u_i$. See for example Burchard-Hajaiej \cite{BH} and the references therein. 

We now describe the properties of multi-marginal mass transport under a compatible function cost $H$. Given probability measures $\mu_1,\mu_2,...,\mu_m$ on $\mathbb{R}$, if $H$-is compatible, then $\frac{\partial ^2 H}{\partial q_i \partial q_j} <0$ for the new variables $q_i$ defined in \eqref{chcoord} (i.e., $H$ is submodular  relative to the $q$ variables). Then as shown by Carlier \cite{C},  there is a unique solution to the optimal transportation problem \eqref{multmarg}, with marginals in these coordinates given by $q_i \# \mu_i$, given by $\gamma =(I,f_2,f_3,...f_m)_{\#}\mu_1$, where $f_i: \mathbb{R} \rightarrow \mathbb{R}$ is the unique increasing map pushing forward $\mu_1$ to $q_i \# \mu_i$. 

In the original $p$ coordinates, then, the unique solution to \eqref{multmarg} is $\gamma = (Id, g_1,g_2,...,g_m)_{\#}\mu_1$, where $g_i: \mathbb{R} \rightarrow \mathbb{R}$, defined by $g_i(p_i) = q_i^{-1}(f_i(p_1))$ is \textit{increasing} if $\frac{\partial ^2 H}{\partial p_1 \partial p_i}<0$ and \textit{decreasing} if $\frac{\partial ^2 H}{\partial p_1 \partial p_i}>0$. It is also the unique such map pushing forward $\mu_1$ to $(\mu_1,\mu_2,...,\mu_m)$.   For proofs and more discussion of the compatibility condition, see \cite{P4},  section 3.5.

We now discuss the notion of $H$-monotonicity introduced by Fazly-Ghoussoub in \cite{FG}.


\newtheorem{mono}[hmono]{Definition}
\begin{mono}\label{mono}

\begin{enumerate}
\item A function $u =(u_1,u_2,...,u_m)\in C^1(\mathbb{R}^N; \mathbb{R}^m)$ is said to be monotone if each $u_i$ is strictly monotone with respect to $x_N$; that is, if $\frac{\partial u_i}{\partial x_N} \neq 0$.
\item $u$ is said to be $H$-monotone if it is monotone and if for all $i \neq j$,
\begin{equation*}
\frac{\partial ^2 H}{\partial p_i \partial p_j} \frac{\partial u_i}{\partial x_N} \frac{\partial u_j}{\partial x_N} <0.
\end{equation*}
\end{enumerate}
\end{mono}

It is easy to see that the existence of an $H$-monotone function $u =(u_1,u_2,...,u_m)$ implies that $H$ is necessarily orientable  on the range $\Omega$ of $u$. More generally, we shall say -- as in \cite{FG}-- that $H$ is orientable at $u =(u_1,u_2,...,u_m)$, if there exist functions $(\theta_i(x))_{i=1}^m$ in $C(\R^N, \R)$, which do not change sign such that
\[
\hbox{$\frac{\partial ^2 H}{\partial p_i \partial p_j}(u(x))\theta_i(x) \theta_j(x) <0$ for all $x\in \R^N$.}
\]
We now note the following 
easy application of the above lemma coupled with Carlier's result. For our purposes it will often be useful to decompose $x \in \mathbb{R}^N$ into $(x', x_N) \in \mathbb{R}^{N-1} \times \mathbb{R}$.

\newtheorem{monooptim}[hmono]{Proposition}
\begin{monooptim}\label{monooptim}
Let $u=(u_1,u_2,...,u_m)$ be a bounded monotone function in $C^1(\mathbb{R}^N; \mathbb{R}^m)$.  Let $\mu$ be a probability measure on $\mathbb{R}$ that is absolutely continuous with respect to Lebesgue measure. For each $x' \in \mathbb{R}^{N-1}$, let $\mu_i^{x'}$ be the pushforward of $\mu$ by the map $u_i^{x'}:x_N \mapsto u_i(x',x_N)$ and set $\gamma^{x'} := (u_1^{x'},u_2^{x'},...,u_m^{x'})_{\#} \mu$.  Then the following are equivalent:
\begin{enumerate}
\item $u$ is $H$-monotone.
\item For each $x' \in \mathbb{R}^{N-1}$, the measure $\gamma^{x'}$ is optimal for the Monge-Kantorovich problem \eqref{multmarg}, when the marginals are given by $\mu_i^{x'}$.
\end{enumerate}
\end{monooptim}
\begin{proof}
$1) \rightarrow 2)$ is obvious. Assuming now $2)$, we note that for each $x' \in \mathbb{R}^{N-1}$, the measures $\mu_i^{x'}$ are absolutely continuous with respect to Lebesgue measure by the (strict) monotonicity of the $u_i$.  The support of the optimizer $\gamma^{x'}$ must be $H$-monotone, by the equivalence of $H$ orientability and $H$ compatibility, combined with  Theorem 3.5.3 in \cite{P4}.  As this support is exactly the image of $(u_1^{x'},u_2^{x'},...,u_m^{x'})$, this implies 

\begin{equation*}
\frac{\partial ^2 H}{\partial p_i \partial p_j} \frac{\partial u_i}{\partial x_N} \frac{\partial u_j}{\partial x_N} <0.
\end{equation*}
As this holds for any $x'$, this completes the proof. \end{proof}

\section{Decoupling systems in the presence of $H$-monotone solutions}

  Fazly and Ghoussoub \cite{FG} conjectured, and proved in dimensions $N \leq 3$, that $H$-monotone solutions of the system of elliptic PDEs on $\mathbb{R}^N$,
\begin{equation}\label{ellsys1}
\Delta u = \nabla H (u),
\end{equation}
where $u=(u_1,u_2,...,u_m)$ represents an $m$-tuple of functions, must be essentially $1$-dimensional; that is, each $u_i$ takes the form $u_i(x) = g_i(a_i\cdot x'-x_N)$ for some $a_i \in \mathbb{R}^{N-1}$.  This is a systems analogue of DeGiorgi's famous conjecture for monotone solutions of the Allen-Cahn equation \cite{GG} \cite{AC}.

They also showed that, for orientable systems in dimension $N=2$, all the components of a {\it stable solution}  $u=(u_1,u_2,...,u_m)$ have common level sets, which also happened to be  hyperplanes. We shall therefore say that the components $(u_i)_{i=1}^m$ have common level sets if for any $i\neq j$  and any $\lambda \in \mathbb{R}$, there exists some $\bar \lambda$ such that
\begin{equation*}
\{x: u_i(x) =\lambda\}=\{x: u_j(x) =\bar \lambda\}.
\end{equation*}
In light of this, the following conjecture seems reasonable:

\newtheorem{commlevset}[hmono]{Conjecture}
\begin{commlevset}\label{commlevset}
Any $H$-monotone solution to the system \eqref{ellsys1} has common level sets.
\end{commlevset}

This can be seen as complementary to the conjecture of Fazly and Ghoussoub, which asserts that --at least in low dimensions (say $N\leq 8$) -- the level sets of $H$-monotone solutions are hyperplanes (but possibly different hyperplanes for each component $u_i$ of the solution).  Note that an interesting conjecture of Beresetycki et al., for the case when $m=2$ and $H(p_1,p_2) = \frac{1}{2}p_1^2p_2^2$ amounts to a combination of these two (ie, $H$-monotone solutions take the form $u_1(x)=U_1(a\cdot x), u_2(x)=U_2(a\cdot x)$ for some common $a \in \mathbb{R}^N$).
Fazly and Ghoussoub proved this conjecture in $N\leq 3$ dimensions \cite{FG}; we will provide some additional support for it in section 4 below, as well as several examples of solutions exhibiting this behaviour in section 5.  We now establish a result on the decoupling of the system \eqref{ellsys1}.

\newtheorem{decouple}[hmono]{Theorem}
\begin{decouple}\label{decouple}
Let $u=(u_1,...,u_m)$ be a bounded  solution to system \eqref{ellsys1}. 
\begin{enumerate}

\item If $u=(u_1,...,u_m)$ is monotone, then there exist functions $V_i(p_i,x')$ such that $u$ solves the system of decoupled equations:
\begin{equation}
\hbox{$\Delta u_i = \frac{\partial V_i}{\partial p_i}(u_i(x), x')$\,\, for $i=1,..., m.$}
\end{equation}
 Furthermore, along the solution, we have
  \begin{equation}
\hbox{$ \sum_{i=1}^m V_i(u_i(x), x') =H(u_1(x),u_2(x),...,u_m(x))$ for $x\in \mathbb{R}^N$}.
 \end{equation}
 \item If  $u=(u_1,...,u_m)$ is $H$-monotone, then for all $p=(p_1,p_2,...,p_m) \in \mathbb{R}^m$,
 \begin{equation}
 \sum_{i=1}^m V_i(p_i, x') \leq H(p_1,p_2,...,p_m).
 \end{equation}
\item If the $u_i$ have common level sets, then the $V_i$ can be chosen to be independent of $x'$, that is $V_i(p_i, x')=V_i(p_i)$.
\end{enumerate}
\end{decouple}
\begin{proof}
1) Fix $x' \in \mathbb{R}^{N-1}$, and define $V_i(\cdot, x')$ on the range of $x_N \mapsto u_i(x',x_N)$ as follows. For $p_i$ in this range, monotonicity ensures the existence of a unique $x_N =x_N(p_i)$ such that $p_i = u_i(x', x_N)$. We can therefore set 
\begin{equation*}
\frac{\partial V_i}{\partial p_i}(p_i, x') = \frac{\partial H}{\partial p_i}(u(x',x_N)).
\end{equation*}
It follows by construction that 
\begin{equation*}
\Delta u_i = \frac{\partial H}{\partial p_i}(u(x',x_N)) =\frac{\partial V_i}{\partial p_i}(u_i, x').
\end{equation*}
Note that, as each $V_i(x',\cdot)$ is defined only up to an arbitrary constant, we can choose the constants so that for $x_N =0$,
\begin{equation*}
 \sum_{i=1}^m V_i(u_i(x', 0), x') -H(u_1(x', 0),u_2(x',0),...,u_m(x',0))=0.
\end{equation*}
Now note that by the definition of $V_i$, we have
\begin{eqnarray*}
&& \frac{\partial}{\partial x_N} \Big(\sum_{i=1}^m V_i(u_i(x', x_N), x') -H(u_1(x',  x_N),u_2(x', x_N),...,u_m(x', x_N))\Big)\hfill\\
  &=&\sum_{i=1}^m \frac{\partial V_i}{\partial p_i}(u_i(x', x_N), x')\frac{\partial u_i}{\partial x_N} -\sum_{i=1}^m \frac{\partial H}{\partial p_i}(u_1(x', x_N),u_2(x',x_N),...,u_m(x',x_N))\frac{\partial u_i}{\partial x_N}\\
 &=&0.
\end{eqnarray*}
 Therefore, we have 
\begin{equation*}
 \sum_{i=1}^m V_i(u_i(x', x_N), x') -H(u_1(x', x_N),u_2(x',x_N),...,u_m(x',x_N))=0
 \end{equation*}
everywhere on $\mathbb{R}^{N}$. 

2) Now suppose that $u$ is $H$-monotone.  For fixed $x' \in \mathbb{R}^{N-1}$, let  $\gamma^{x'}$ and $\{\mu_i^{x'}, i=1,...,m\}$ be defined as in Proposition \ref{monooptim}. The measure $\gamma^{x'}$ is then optimal for the Monge-Kantorovich problem \eqref{multmarg} with given  marginals $\mu_i^{x'}, i=1,...,m$.  In this case, the $V_i$ defined above play the role of Kantorovich potentials, and so we have 
\begin{equation*}
 \sum_{i=1}^m V_i(p_i, x') \leq H(p_1,p_2,...,p_m).
 \end{equation*}
everywhere.

3) In this case, the  image of the map $(u_1^{x'},u_2^{x'},...,u_m^{x'})$ is independent of $x'$ .  The measures $\gamma^{x'}$ are then all supported on the same set, and so we can choose the $V_i(p_i,x') =V_i(p_i)$ to be independent of $x'$. The image of $(u_1,u_2,...,u_m)$ is equal to the image of $(u_1^{x'},u_2^{x'},...,u_m^{x'})$ for any fixed $x'$, and the results in \cite{P4} imply that any measure supported on this set is optimal for its marginals. 
\end{proof}
Note that the preceding proof yields an explicit expression for the $V_i$'s:
\begin{equation}
V_i(\tilde p_i,x') = \int_0^{\tilde p_i}  \frac{\partial H}{\partial  p_i}(u(x',x_N(p_i)))dp_i + H(u(x',0)).
\end{equation}

\newtheorem{oned}[hmono]{Remark} \rm 
\begin{oned}\rm
If $u$ is $H$-monotone such that all its components  $(u_i)_i$ have common level sets,  then the image of $(u_1,u_2,...,u_m)$ is a $1$-dimensional set. By $H$-monotonicity, the results in \cite{P4} imply that any measure supported on this set is optimal for its marginals.  In particular, for any measure $\mu$ on $\mathbb{R}^N$, absolutely continuous with respect to Lebesgue measure, the measure $\gamma = (u_1,u_2,...,u_m)_{\#}\mu$ is optimal for its marginals $\mu_i = u_{i\#}\mu$. 
\end{oned}
As an immediate application of Theorem \ref{decouple}, we obtain the following Modica type estimate.

\newtheorem{modica}[hmono]{Corollary}
\begin{modica}
Suppose that $u$ is a bounded monotone solution of \eqref{ellsys1} such that all its components  $(u_i)_i$ have common level sets.  Then, letting $\bar V_i = \min_{p_i \in \text{Range}(u_i)} V_i(p_i)$, we have

\begin{equation}
\frac{1}{2}\sum_{i=1}^m|\nabla u_i(x_i)|^2 \leq H(u) -\sum_{i=1}^m\bar V_i.
\end{equation} 
\end{modica}
\begin{proof}
By Theorem \ref{decouple}, $u$ solves the decoupled system
$\Delta u_i = \frac{ \partial V_i}{\partial p_i}(u_i).$
By Modica's well known inequality for single equations, we have $\frac{1}{2}|\nabla u_i(x_i)|^2 \leq  V_i(u_i)-\bar V_i$.  Summing over $i$ and using Theorem \ref{decouple}, part 1), completes the proof.  
\end{proof}
\newtheorem{dim2}[hmono]{Remark}
\begin{dim2}\rm
This corollary applies to, for example, bounded, monotone solutions in $\mathbb{R}^2$ with an orientable $H$.  Fazly and Ghoussoub \cite{FG} showed that in this case bounded stable \footnote{See \cite{FG} for the definition of stability.} solutions satisfy  
\begin{equation}
\nabla u_j  =C_{i,j}\nabla u_i
\end{equation} 
for constants $C_{i,j}$ with signs opposite the signs of $\frac{\partial ^2 H}{\partial p_i \partial p_j}$ \cite{FG}.  When the level sets of $u_i(x',x_N)=g_i(a_i\cdot x' -x_N)$ are hyperplanes, this implies that $a_i= a_j:=a$ for all $i,j$. Therefore, the $u_i$'s share the same level sets and the corollary applies.

In fact, in this case we can prove even more; we can remove the $\bar V_i$ terms from the inequality.  A simple computation yields 
\begin{equation*}
(|a|^2+1)g_i''(a\cdot x' -x_N) =\Delta u_i(x)=V_i'(g_i(a\cdot x' -x_N)).
\end{equation*}
Denoting the argument $a\cdot x' -x_N$ by $y$, multiplying the equation by $g'(y)$ and integrating, we obtain

\begin{equation}\label{1dmodica}
\frac{1}{2}(|a|^2+1)g_i'(y)^2 = V_i(g_i(y)) -C_i,
\end{equation} 
for some constants $C_i$.  As the solution $g_i(a\cdot x' -x_N)$ is bounded and monotone, we can take the limit as $y \rightarrow \infty$; as we must have $\lim_{y \rightarrow \infty} g'(y) =0$, we obtain $C_i = V_i(\lim_{y \rightarrow \infty}g_i(y))$.  Now, summing over $i$ in \eqref{1dmodica} yields

\begin{equation}\label{1dmodicasys}
\frac{1}{2}(|a|^2+1)\sum_{i=1}^m g_i'(y)^2 = \sum_{i=1}^mV_i(g_i(y)) -\sum_{i=1}^mC_i =H(g_1(y),g_2(y),...,g_m(y))-\sum_{i=1}^mC_i.
\end{equation} 
Taking the limit as $y \rightarrow \infty$ of the equality 
\begin{equation*}
H(g_1(y),g_2(y),...,g_m(y)) =\sum_{i=1}^mV_i(g_i(y))
\end{equation*}
yields 
\begin{equation*}
H(\lim_{y \rightarrow \infty}g_1(y),\lim_{y \rightarrow \infty}g_2(y),...,\lim_{y \rightarrow \infty}g_m(y)) =\sum_{i=1}^mC_i.
\end{equation*}
If $H$ is nonnegative, \eqref{1dmodicasys} yields

\begin{eqnarray*}
\frac{1}{2}\sum_{i=1}^m |\nabla u_i(x_i)|^2&=&\frac{1}{2}(|a|^2+1)\sum_{i=1}^m g_i'(y)^2\\
 &\leq& H(g_1(y),g_2(y),...,g_m(y))\\
&=& H(u_1(x),u_2(x),...,u_m(x)).
\end{eqnarray*} 
\end{dim2}
\newtheorem{antimono}[hmono]{Remark}
\begin{antimono}\rm
It is not clear to us whether the approach above provides any more information about $H$-monotone solutions when the solutions are not $1$-dimensional, but surprisingly, it does imply more about solutions with the {\it opposite} geometry. Indeed, suppose that $-H$ is orientable, and that $u$ is a bounded $-H$ monotone solution to \eqref{ellsys1}; in other words, for all $i \neq j$

\begin{equation*}
\frac{\partial ^2 H}{\partial p_i \partial p_j} \frac{\partial u_i}{\partial x_N} \frac{\partial u_j}{\partial x_N}>0.
\end{equation*}
Then for each $x'$, the measure $\gamma^{x'}$ {\it maximizes} (rather than minimizes) the multi-marginal Kantorovich functional \eqref{multmarg} and the functions $V_i$ from Theorem \ref{decouple} satisfy the constraint 
\begin{equation}\label{potentials}
\sum V_i(x',p_i) \geq H(p_1,p_2,...,p_m).
\end{equation}
If the $u_i$ also have common level sets, then the $V_i$ are independent of $x'$, by Theorem \ref{decouple}.  If $H$ is everywhere nonnegative, taking infimums in \eqref{potentials} yields $\sum_{i=1}^m \bar V_i \geq 0$; combining this with the previous corollary implies the following Modica estimate for these types of solutions:

\begin{equation*}
\frac{1}{2}\sum_{i=1}^m|\nabla u_i(x_i)|^2 \leq H(u).
\end{equation*} 
\end{antimono}
\section{One dimensional solutions with common level sets and energy minimizers on finite domains}
In this section, we exploit the connection with optimal transport to prove that, under certain conditions, minimizers of the energy $E$ on bounded domains have common level sets and are $1$-dimensional  

The proposition below is more closely related to a Gibbons type conjecture than one of De Giorgi type\footnote{The Gibbons type conjecture for $H$-monotone solutions is a DeGiorgi conjecture with an additional assumption that the functions $u_i$ converge uniformly in $x'$ as $x_N \rightarrow \pm \infty$.  Under this condition, the boundary conditions in Theorem \ref{bounded_domains} are almost satisfied, if $[0,1]$ is replaced with a large enough interval, for any $\Omega$}.  We feel this provides some explanation for why $H$-monotone solutions seem to share common level sets as well as being $1$-dimensional, as conjectured in Conjecture \ref{commlevset} here and Conjecture 2 in \cite{FG}, respectively (and proven for stable solutions when $N \leq 2$ in \cite{FG}).

Heuristically, it is the $H$-term in the energy that forces the $u_i$ to have the same level sets, so the image of $(u_1,....u_m)$ is optimal for its marginals in \eqref{multmarg}.  It is the Dirichlet term that then forces these level sets to be hyperplanes.  For a single equation, one dimensional rearrangements reduce the Dirichlet energy because of Polya-Sz\"ego type arguments.  On the other hand, the term $\int H(u_1)dx$ is unchanged by rearrangement, as $u_1 \# dx$ remains the same.  For systems, $\int H(u_1,u_2,...,u_m)dx$ will be lowered if the appropriate rearrangements $\bar u_i$ can be made so that $(\bar u_1,...,\bar u_m)\#dx$ solves the optimal transport problem \eqref{multmarg} with marginals $u_i \# dx$.  For an orientable $H$, this is possible.  For a non-orientable $H$, the solution to the optimal transport problem may concentrate on a higher dimensional set. 

Let $\Omega \subset \mathbb{R}^{N-1}$ be a bounded domain.  For $u=(u_1,...u_m) :\Omega \times[0,1] \rightarrow \mathbb{R}^m$, set

\begin{equation*}
E(u) =\int_{\Omega \times [0,1]}\frac{1}{2}\sum_{i=1}^m|\nabla u_i|^2+H(u)dx
\end{equation*}

\newtheorem{bounded_domains}{Theorem}[section]
\begin{bounded_domains}\label{bounded_domains}
Suppose $H$ is orientable  and that the set $\{a_i,b_i\}_{i=1}^m$ of constants is consistent with the orientability condition,  that is
\begin{equation}
\hbox{$H_{i,j}(u)(b_i-a_i)(b_j-a_j) <0$ for all $i \neq j$.}
\end{equation}
Assume $u$ minimizes the energy $E$ in the class of functions $v=(v_1,...v_m) :\Omega \times[0,1] \rightarrow \mathbb{R}^m$ satisfying
\begin{enumerate}
\item $v_i(x',0) = a_i$, $v_i(x',1) =b_i$,
\item $v$ is $H$-monotone with respect to $x_N$.
\end{enumerate}
Then, $u(x',x_N) = u(x_N)$ is one dimensional and the components $u_i's$ have common level sets.
\end{bounded_domains}
\begin{proof}
Suppose $u$ is a minimizer satisfying the assumptions in the theorem.  Set $\mu_i = u_i\# \mu$, where $\mu$ is Lebesgue measure on $\Omega \times [0,1]$.  If $a_i < b_i$, define $\bar u_i(x) =\bar u_i(x_N)$ implicitly by
\begin{equation*}
\mu(\Omega \times [0,x_N]) = \mu_i(a_i, \bar u_i(x_N));
\end{equation*}
note that this is eqivalent to the explicit formula 

\begin{equation*}
\bar u_i(x_N) =\inf \{y: \mu(\Omega \times [0,x_N]) \leq \mu_i(a_i,y )\}.
\end{equation*}

If $a_i >b_i$, define $\bar u_i(x) =\bar u_i(x_N)$ implicitly by
\begin{equation*}
\mu(\Omega \times [0,x_N]) = \mu_i(\bar u_i(x_N), a_i);
\end{equation*}
We will refer to $\bar u = (\bar u_1, \bar u_2,...,\bar u_m)$ as the $H$-monotone rectangular rearrangment of $u$.  It is clear that $(\bar u_1, ....,\bar u_m) \# \mu$ is the unique solution to the optimal transport problem with marginals $\mu_i$ and cost $H$, as its support is a $1$-dimensional $H$-monotone set.  Therefore, 
\begin{equation*}
\int H(u_1,...u_m)dx \geq \int H(\bar u_1,...\bar u_m)dx
\end{equation*}
and, by uniqueness in the optimal transport problem, we can have equality only if the $u_i$ have common level sets, in which case $(u_1, ...., u_m) \# \mu=(\bar u_1, ....,\bar u_m) \# \mu$

To complete the proof, we need only to show that $\int_{\Omega \times[0,1]}|\nabla \bar u_i|^2dx \leq \int_{\Omega \times[0,1]}|\nabla u_i|^2 dx$, and we can have equality only if the level sets of $u_i$ are vertical hyperplanes.  We establish this in a separate Lemma below. 
\end{proof}
\newtheorem{rectpz}[bounded_domains]{Lemma}
\begin{rectpz}
Suppose $u_i:\Omega \times [0,1] \rightarrow \mathbb{R}$ satisfies 
\begin{enumerate}
\item $u_i(x',0) = a_i$, $u_i(x',1) =b_i$,
\item $\frac{\partial u_i}{\partial x_N} >0$.
\end{enumerate}
Define the rectangular rearrangement $\bar u_i$ of $u_i$ via 

\begin{equation*}
\mu(\Omega \times [0,x_N]) = \mu_i(a_i, \bar u_i(x_N)),
\end{equation*}
where $\mu_i = u_i\# \mu$, and $\mu$ is Lebesgue measure on $\Omega \times [0,1]$.
Then,
\begin{equation*}
\int_{\Omega \times[0,1]}|\nabla \bar u_i|^2dx \leq \int_{\Omega \times[0,1]}|\nabla u_i|^2 dx.
\end{equation*}
\end{rectpz}
The proof is adapted from the proof of Theorem 4.7 (the Polya-Sz\"ego inequality) in the lecture notes of Burchard \cite{bur}.
\begin{proof}
By the co-area formula,

\begin{equation*}
\int_{\Omega \times[0,1]}|\nabla u_i|^2 dx =\int_a^b\int_{u_i^{-1}(t)}\int|\nabla u_i|dH^{N-1}dt,
\end{equation*}
where $H^{N-1}$ denotes $N-1$ dimensional Hausdorff measure.  Now, $s \mapsto s^{-1}$ is convex, and so, by Jensen's inequality, we have
\begin{equation*}
\Big(\int_{u_i^{-1}(t)} |\nabla u_i|^{-1}\frac{dH^{N-1}}{H^{N-1}(u_i^{-1}(t))}\Big)^{-1} \leq \int_{u_i^{-1}(t)} |\nabla u_i|\frac{dH^{N-1}}{H^{N-1}(u_i^{-1}(t))}.
\end{equation*}
Now, as the gradients of both $u_i$ and $\bar u_i$ never vanish, and the functions are equimeasurable, we have as an easy consequence of the co-area formula

\begin{equation*}
\int_{u_i^{-1}(t)} |\nabla u_i|^{-1}dH^{N-1} =\int_{u_i^{-1}(t)} |\nabla \bar u_i|^{-1}dH^{N-1}.
\end{equation*}
Finally, note that for $t \in [a,b]$, $H^{N-1}(\{\bar u_i =t\}) = H^{N-1}(\Omega)$.  On the other hand, for every point $x'$ in $\Omega$, we have $u_i(x',0) =a$ and $u_i(x',1) =b$; it follows that there is some $x_N$ so that $u_i(x',x_N) =t$.  Therefore, the projection $(x',x_N) \mapsto x'$ from the set $\{ u_i =t\}$ to $\Omega$ is surjective.  As this projection clearly has Lipschitz norm $1$, it follows that $H^{N-1}(\{ u =t\}) \geq H^{N-1}(\Omega)$.  We therefore have
\begin{equation*}
H^{N-1}(\{\bar u_i =t\}) \leq H^{N-1}(\{ u_i =t\}).
\end{equation*}
It now follows that 
\begin{eqnarray*}
\int_{u_i^{-1}(t)} |\nabla u_i|dH^{N-1}& \geq & \big(H^{N-1}(u_i^{-1}(t))\big)^2\Big(\int_{u_i^{-1}(t)} |\nabla u_i|^{-1}dH^{N-1}\Big)^{-1}  \\
& \geq &  \Big(H^{N-1}({\bar u_i}^{-1}(t))\big)^2\Big( \int_{{\bar u_i}^{-1}(t)} |\nabla {\bar u_i}|^{-1}dH^{N-1} \Big)^{-1}.
\end{eqnarray*}
Now, as $\nabla \bar u_i$ is constant on the level sets of $\bar u_i$, we get equality in Jensen's inequality when $u_i = \bar u_i$; therefore, we get equality in the above sequence of inequalities when $u_i =\bar u_i$, completing the proof.
\end{proof}

\section{Examples of decoupled systems}
\subsection{Allen-Cahn potentials with quadratic interaction}
Consider the nonlinearity $H(u) = \sum_{i \neq j}|u_i -u_j|^2 +\sum_{i=1}^mW(u_i)$, where $W(u_i)=\frac{1}{4}(u_i^2-1)^2$ is the Allen-Cahn potential. Since $\frac{\partial ^2 H}{\partial u_i \partial u_j} = -2 <0$ for all $i \neq j$, $H$ is orientable.  It is clear that 
\begin{equation}\label{decoupcond}
H(u) \geq \sum_{i=1}^mW(u_i),
\end{equation}
with equality when $u_i =u_j$ for all $i \neq j$.  Therefore, taking $V_i(u_i) = W(u_i)$, we obtain the decoupled system
\begin{equation*}
\Delta u_i = V_i'(u_i) =W'(u_i) = u_i(u_i^2-1).
\end{equation*}
It is well known that the one dimensional solutions to this equation are of the form $u_i(x) = \tanh(\frac{a\cdot x -b}{\sqrt{2}})$, for constants $b \in \mathbb{R}$, $a \in \mathbb{R}^n$, with $|a| =1$; note that these functions, with a common $a,b$  solve the decoupled system {\it and} also satisfy the constraint $u_i=u_j$, so that we have equality in \eqref{decoupcond}.  We therefore obtain a one dimensional solution, with common level sets, to the original coupled system

\begin{equation*}
\Delta u_i = H'(u_i) = 2(m-1)u_i - 2\sum_{j \neq i}^mu_j + u_i(u_i^2-1).
\end{equation*}
(one could also easily verify directly that these functions solve  the original system).
We do not know whether there are other $H$-monotone solutions to this equation. If so, the result of Fazly-Ghoussoub \cite{FG} then implies that, at least in low dimensions, their components are one dimensional and have common level sets. 
\subsection{Products of Allen-Cahn potentials}

Consider the nonlinearity $H(u_1,u_2,...u_m) = m\log[\frac{1}{m}\sum_{i=1}^me^{\frac{1}{4}(u_i^2-1)^2}]$.  Assuming $|u_i| \leq 1$, applying the  arithmetic-geometric mean inequality to the quantities $e^{\frac{1}{4}(u_i^2-1)^2}$ and taking logarithms of both sides implies
\begin{equation}\label{decoupcond2}
\sum_{i=1}^m\frac{1}{4}(u_i^2-1)^2\leq H(u_1,u_2,...u_m),  
\end{equation}
with equality when $u_i^2 =u_j^2$ for all $i,j$.  As in the previous example, the decoupled system becomes
\begin{equation}
\Delta u_i = V_i'(u_i) =W'(u_i) = u_i(u_i^2-1).
\end{equation}
Note that, for this nonlinearity the mixed second order partial derivatives change signs, as 
\begin{equation*}
\frac{\partial ^2 H}{\partial u_i \partial u_j} = \frac{-mu_iu_j (u_i^2-1)(u_j^2-1)e^{\frac{1}{4}(u_i^2-1)^2}e^{\frac{1}{4}(u_j^2-1)^2}}{[\sum_{k=1}^me^{\frac{1}{4}(u_k^2-1)^2}]^2}.
\end{equation*}
However, it is easy to see that $H$ is orientable in any region where $u_i \neq 0$ and $|u_i| \leq 1$.  Here, we can take as solutions to the decoupled problem 
\begin{equation*}
u_i(x) = \pm \tanh(\frac{a\cdot x -b}{\sqrt{2}})
\end{equation*}
as with either sign, the functions solve the decoupled system and satisfy the equality constraint in \eqref{decoupcond2}.  It is also straightforward to check directly that these functions solve the original system.
\subsection{Coupled quadratic system}
In \cite{blwz}, Berestycki, Lin, Wei and Zhao  studied the case where $H(u_1,u_2) = \frac{1}{2}u_1^2u_2^2$, corresponding to the system 
\begin{eqnarray}\label{blwz}
\Delta u_1 &=&u_1u_2^2\nonumber\\
\Delta u_2 &=&u_1^2u_2,
\end{eqnarray}
along with the extra constraint, $u_1, u_2 \geq 0$.   Here, $H$-conjugate Kantorovich potentials amount to increasing concave functions with increasing Legendre duals; that is,

\begin{equation*}
V_1(p_1) = \inf_{p_2} \frac{1}{2}p_1^2p_2^2 - V_2(p_2)
\end{equation*}
for some $V_2:[0, \infty) \rightarrow \mathbb{R}$ if and only if $q_1 \mapsto F_1(q_1):=2V_1(\sqrt{q_1}) =\inf_{q_2>0} q_1q_2 - 2V_2(\sqrt{q_2})$ is concave and $V_1'(q_1) \geq 0$ (as $V_1'(q_1) =q_2$ for a minimizing $q_2$).

Finding solutions with common level sets to the system (\ref{blwz}) can then be reformulated as looking for a concave function $F_1$, with conjugate $F_2$, such that both $F_1$ and $F_2$ are increasing on $[0,\infty)$, as well as non-negative functions $u_1,u_2$, saturating everywhere the inequality $F_1(u_1^2) +F_2(u_2^2) \leq u_1^2u_2^2$, and solving the decoupled system
\begin{eqnarray*}
\Delta u_1 =2F_1(u_1^2) \\
\Delta u_2 =2F_2(u_2^2).
\end{eqnarray*}
Any such solution $(u_1, u_2)$ is automatically an $H$-monotone solution to the original system, with common level sets.  In low dimensions, DeGiorgi type results for scalar equations imply that such solutions are one dimensional.  Below, we exhibit an explicit example of such a solution.

In one dimension, Berestycki et al. found a solution with 
\begin{eqnarray*}
u_1(-\infty) =0,& & u_1(\infty) =\infty,\\
 u_2(-\infty) =\infty,& & u_2(-\infty) =0.
\end{eqnarray*}
We can use this solution to build decoupling potentials as follows. For all $p_1 \geq 0$, set $V_1(0)=0$ and $V_1'(p_1) = p_1u_2^2(x)$ for the unique $x$ such that $p_1 = u_1(x)$ and define $V_2(p_2)$ analogously.  In other words,
\[
\hbox{$V_1(t)=\int_0^tsu_2^2(u_1^{-1}(s))\, ds$ and $V_2(t)=\int_0^tsu_1^2(u_2^{-1}(s))\, ds.$}
\]
 The functions $F_i(q_i) =2V_i(\sqrt{q_i})$ are then concave conjugates, while $u_1$ and $u_2$ solve the decoupled system
\begin{equation*}
\Delta u_i = V_i'(u_i) = u_iF'(u_i^2).
\end{equation*}
In light of the preceding results, we suspect that in low dimensions, any $H$-monotone solution with appropriate limits must satisfy this decoupled system. 

Finally, let us emphasise that this gives a new method for finding solutions to this system.  One looks first for an increasing concave function $F_1$, and a solution $u_1$ to the equation $\Delta u_1 =2F_1(u_1^2)$.  This will yield a solution to the system, provided   $u_2$ solves $\Delta u_2 =2F_2(u_2^2)$, where $F_2$ is the concave conjugate of $F_1$ and for $u_2 = \sqrt { F_1'(u_1^2)}$ (this last condition ensures that $F_1(u_1^2) +F_2(u_2^2) =u_1^2u_2^2$).

\bibliographystyle{plain}

\begin{thebibliography}{99}

\bibitem {Ali} N. D. Alikakos, {\it Some basic facts on the system  $\Delta u-W_u(u)=0$}. 
 Proc. Amer. Math. Soc. 139 (2011), 153-162.
 
 \bibitem{AC} L. Ambrosio and X. Cabr\'{e}, {\it Entire solutions of semilinear elliptic equations in $\mathbf R^3$ and a conjecture of De Giorgi}, J. Amer. Math. Soc. 13 (2000), 725-739. 
 
\bibitem{blwz} H. Berestycki, T.-C. Lin, J. Wei and C. Zhao, {\it On Phase-Separation Models: Asymptotics and Qualitative Properties}, Archive for Rational Mechanics and Analysis, 208, (2013), 163-200.
\bibitem{bur} A. Burchard, {\it A Short Course on Rearrangement Inequalities}, lecture notes (2009).
\bibitem{BH}A. Burchard and H. Hajaiej, {\it Rearrangement inequalities for functionals with monotone integrands},  Journal of Functional Analysis, 233(2) p. 561-582 (2006).



\bibitem{C} G. Carlier, {\it On a class of multidimensional optimal transportation problems}, J. Convex Anal., 10(2), p. 517-529 (2003).

\bibitem{FG} M. Fazly and N. Ghoussoub, {\it De Giorgi type results for elliptic systems,} Calculus of Variations and Partial Differential Equations, Vol. 47, Issue 3 (2013), p. 809-823.

\bibitem{GG} N. Ghoussoub and C. Gui, {\it On a conjecture of de Giorgi
and some related problems},  Math. Ann. Vol. {\bf 311} (1998) p. 481-491.

\bibitem{GG2} N. Ghoussoub, C. Gui, {\it On the De Giorgi conjecture in dimensions 4 and 5},  Annals of Mathematics, 157 (2003)  p. 313-334.  

\bibitem{HS} H. Hajaiej and C. A. Stuart, {\it Extensions of the Hardy-Littlewood inequalities for Schwarz symmetrization},  Int. J. Math. Math. Sci., (57-60), p. 3129-3150 (2004).
\bibitem{lorentz} G. G. Lorentz, {\it An inequality for rearrangements},  Amer. Math. Monthly, 60, p. 176-179 (1953).
\bibitem{M} L. Modica, {\it A gradient bound and a Liouville theorem for non 
linear Poisson eqautions},  Comm. Pure. Appl. Math. Vol. { \bf 38} 
(1985), p. 679-684. 

\bibitem{P4} B. Pass, {\it Structural results on optimal transportation plans},  PhD thesis, University of Toronto, 2011. Available at http://www.ualberta.ca/ pass/thesis.pdf.

\end{thebibliography}

\end{document}